\documentclass[reqno]{amsart}
\usepackage{amssymb,latexsym,amsmath,amsthm,enumerate}
\usepackage{mathtools}
\usepackage{cite}
\usepackage[hidelinks]{hyperref}
\hypersetup{
 pdftitle={Strong nonfinite basability of the three-element flat semiring S7},
 pdfauthor={Mengya Yue and Xiaolei Shao}
}
\makeatletter
\@ifundefined{subjclassname@2020}{%
 \@namedef{subjclassname@2020}{\textup{2020} Mathematics Subject Classification}}{}
\makeatother
\theoremstyle{plain}
\newtheorem{theorem}{Theorem}[section]
\newtheorem{lemma}[theorem]{Lemma}
\newtheorem{proposition}[theorem]{Proposition}
\newtheorem{corollary}[theorem]{Corollary}
\theoremstyle{definition}

\newcommand{\Var}{\mathsf V}
\newcommand{\NN}{\mathbb N_0}
\newcommand{\HH}{\mathcal H}

\newcommand{\cont}{\operatorname{c}}
\newcommand{\len}[1]{\lVert #1\rVert_1}
\newcommand{\Pow}{\mathcal P^{+}}
\newcommand{\Id}{\operatorname{Id}}
\newcommand{\1}{\mathbf 1}

\renewcommand{\ge}{\geqslant}
\renewcommand{\le}{\leqslant}

\title[Strong nonfinite basability of $S_7$]{Strong nonfinite basability of the three-element flat semiring $S_7$}
\date{}

\author{Mengya Yue}
\address{School of Mathematics, Northwest University, Xi'an, 710127, Shaanxi, P.R. China}
\email{myayue@yeah.net}

\author{Xiaolei Shao}
\address{School of Mathematics, Northwest University, Xi'an, 710127, Shaanxi, P.R. China}
\email{xiaoleishao@yeah.net}


\begin{document}

\begin{abstract}
We prove that the three-element flat semiring $S_7$ is strongly
nonfinitely based within the class of additively idempotent semirings.
More generally, let $\mathcal N$ be the subvariety of $\Var(S_7)$
defined by $x^2y\approx x^2$. For every finite ai-semiring $A$,
each variety $\mathcal W$ satisfying
$\mathcal N\subseteq\mathcal W\subseteq\Var(A)$ has infinite
axiomatic rank.
\end{abstract}

\maketitle

\section{Introduction}\label{sec:intro}

An \emph{additively idempotent semiring}, or \emph{ai-semiring}, is an
algebra $(S,+,\cdot)$ in which $(S,+)$ is a semilattice, $(S,\cdot)$ is
a semigroup, and multiplication distributes over addition on both
sides. Until Section~\ref{sec:constants}, all varieties considered
have the binary signature $\{+,\cdot\}$. No additive identity or multiplicative identity is
required, and no constant symbols are included. We write $\Var(A)$ for
the variety generated by an algebra $A$.

An algebra is \emph{finitely based} if all its identities follow from
a finite set of identities. The finite basis problem asks which
algebras have this property. For finite semigroups, Perkins~\cite{Perkins}
gave some of the first nonfinitely based examples. In the broader
setting of finite algebras, McKenzie~\cite{McKenzie} proved that the
finite basis problem is undecidable. For semirings, early results
include the absence of a finite basis for the max-plus algebra of the
natural numbers~\cite{AEI} and finite nonfinitely based semirings
constructed by Dolinka~\cite{D07,D09inherent,D09remark}.

The distinction between the identities of a semiring and those of its
multiplicative reduct is central to the present paper. Jackson, Ren
and Zhao~\cite{JRZ} constructed finite nonfinitely based ai-semirings
whose multiplicative reducts are finitely based. Among them is the
three-element commutative semiring $S_7$, with operations
\[
\begin{array}{c|ccc}
 +&1&a&0\\\hline
 1&1&0&0\\
 a&0&a&0\\
 0&0&0&0
\end{array}
\qquad
\begin{array}{c|ccc}
 \cdot&1&a&0\\\hline
 1&1&a&0\\
 a&a&0&0\\
 0&0&0&0
\end{array}.
\]
Here $1$ is a multiplicative identity and $a^2=0$. The element $0$
absorbs both operations; in particular, it is the greatest element of
the additive semilattice, rather than an additive identity. The symbols
$1$ and $0$ name elements of this algebra, not symbols in the signature.
Its addition is \emph{flat}: any two distinct elements have sum $0$.

The semiring $S_7$ is the smallest possible example of this behaviour.
The variety generated by all two-element ai-semirings is hereditarily
finitely based~\cite{SR}. The analysis of three-element ai-semirings
in~\cite{ZRCS}, together with the proof for $S_7$ in~\cite{JRZ}, shows
that $S_7$ is the unique nonfinitely based three-element ai-semiring up
to isomorphism. The order-four problem has led to classifications for
several additive semilattice types~\cite{RLZC,YRZS,RLYC} and to further
individual nonfinitely based examples~\cite{YRfour,YRGfour}.

The finite basis problem has also been studied for endomorphism semirings
of finite semilattices with zero~\cite{D09endo}, finite chains, and
arbitrary finite semilattices~\cite{GVchains,DGV}.
For naturally ordered Brandt semigroups, the ai-semiring based on the
Brandt monoid is nonfinitely based~\cite{VolkovBrandt}, whereas the
related semirings based on $B_0$ and $B_2$ have finite
bases~\cite{ShapB0,ShapB2}.

A finite ai-semiring $B$ is \emph{strongly nonfinitely based} if every
finite ai-semiring $A$ for which $B\in\Var(A)$ is nonfinitely based.
The question whether $S_7$ has this property is stated explicitly as
Problem~1.1 in Gao, Ren and Zhao~\cite{GRZcontinuum} and as Problem~1.2
in Yue, Ren and Gao~\cite{YRGpower}. Several special cases are known.
Wu, Ren and Zhao~\cite{WRZ} proved that adjoining an additive identity
to $S_7$ gives a nonfinitely based semiring $S_7^0$. Further results
on this adjunction appear in~\cite{ZW}. Gao, Ren and
Zhao~\cite{GRZcontinuum} studied intervals above $\Var(S_7)$, including
the continuum of varieties between $\Var(S_7)$ and $\Var(S_7^0)$.
We prove that $S_7$ is strongly nonfinitely based in
Corollary~\ref{thm:main}.

Flat extensions of groups and limit varieties were studied
in~\cite{RJZL}. Here a
\emph{limit variety} is a nonfinitely based variety all of whose proper
subvarieties are finitely based. A further limit variety generated by
a six-element ai-semiring was obtained in~\cite{LRYlimit}. In an
unpublished manuscript, Yue~\cite{YueSub} classifies the finitely based
subvarieties and the limit subvarieties of $\Var(S_7^0)$. These
questions concern subvarieties of a fixed variety. The main theorem
applies to intervals with a fixed lower endpoint and an upper endpoint
generated by any finite ai-semiring.

Put
\begin{equation}\label{eq:nilvariety}
 \mathcal N=\Var(S_7)\cap\operatorname{Mod}(x^2y\approx x^2),
\end{equation}
where $\operatorname{Mod}(\Sigma)$ denotes the class of ai-semirings
satisfying the identities in $\Sigma$. This is the nil subvariety
studied in~\cite{GJRZ}. It is a proper subvariety of $\Var(S_7)$,
since its defining identity fails in $S_7$ at $x=1$ and $y=a$.
It is also important to distinguish $\mathcal N$ from
$\Var(S_c(abc))$. For a word $a_1\cdots a_k$ with distinct letters,
$S_c(a_1\cdots a_k)$ is the flat semiring of its nonempty commutative
divisors, together with an absorbing element $0$; a product is $0$
when it is no longer a divisor. One has
\[
 \Var(S_c(abc))\subsetneq\mathcal N\subsetneq\Var(S_7).
\]
Indeed, all the semirings $S_c(a_1\cdots a_k)$ belong to $\mathcal N$
by~\cite{GJRZ}; the identity $xyzt\approx x^2$ holds in $S_c(abc)$
and fails in $S_c(abcd)$. Related finite basis results for flat
nil-semirings and flat semirings with nilpotent multiplicative reducts
appear in~\cite{WZR,GRnil}.

A variety has \emph{infinite axiomatic rank} if, for every positive
integer $n$, its identities involving at most $n$ variables do not
form a basis for all its identities. This implies that the variety is
nonfinitely based. For varieties generated by finite algebras in a
finite signature, finite axiomatic rank is equivalent to finite
basability; see Appendix~\ref{sec:rank}.

\begin{theorem}\label{thm:nilinterval}
Let $A$ be a finite ai-semiring and let $\mathcal W$ be an ai-semiring
variety such that
\[
 \mathcal N\subseteq\mathcal W\subseteq\Var(A).
\]
For every positive integer $n$, there is a finite commutative flat
semiring $B_n$, with nilpotent multiplicative reduct and satisfying
$x^2\approx y^2$, such that
\begin{enumerate}[$(1)$]
\item every subsemiring of $B_n$ generated by at most $n$ elements
belongs to $\mathcal N$;
\item $B_n\notin\Var(A)$.
\end{enumerate}
Thus $B_n$ satisfies every identity of $\mathcal W$ in at most $n$
variables but does not belong to $\mathcal W$. In particular,
$\mathcal W$ has infinite axiomatic rank.
\end{theorem}

\begin{corollary}\label{thm:main}
Let $A$ be a finite ai-semiring. If $S_7\in\Var(A)$, then $\Var(A)$
has infinite axiomatic rank. Consequently, $S_7$ is strongly
nonfinitely based.
\end{corollary}

\begin{proof}
The hypothesis gives
$\mathcal N\subseteq\Var(S_7)\subseteq\Var(A)$.
Apply Theorem~\ref{thm:nilinterval} with $\mathcal W=\Var(A)$.
Every finite identity basis uses only finitely many variables, so
infinite axiomatic rank excludes such a basis.
\end{proof}

For comparison, Gao, Jackson, Ren and Zhao~\cite{GJRZ} proved that
every variety in the interval
$[\Var(S_c(abc)),\Var(S_7^0)]$ is nonfinitely based. Its lower
endpoint is smaller than $\mathcal N$. Theorem~\ref{thm:nilinterval}
permits an arbitrary finite ai-semiring at the upper endpoint, subject
to containment of $\mathcal N$. In particular, it implies infinite
axiomatic rank for every variety between $\Var(S_7)$ and any
ai-semiring variety generated by a finite algebra and containing $S_7$. The known
results for $S_7^0$ therefore follow at the level of nonfinite basability.
Their descriptions of explicit bases
and subvariety lattices remain separate results.

The proof constructs, for each $n$, a finite semiring that belongs
locally to $\mathcal N$ but fails an identity of $A$. The construction
starts from a uniform hypergraph. Bounded intersections of distinct
edges allow us to control subsemirings generated by at most $n$
elements using the transversal description of the identities of
$S_7$ from~\cite[Proposition~5.5]{JRZ}. A probabilistic construction
then supplies the required extensions of prescribed vertex sets and
colour multiplicities. To obtain the separating identity, we regard
words with a specified product in the finite multiplicative reduct of
$A$ as a regular language. Semilinearity of the corresponding
multiplicity vectors gives bounds independent of the length of the
words. The separating polynomial includes every ordering of each
hyperedge, which allows this step to apply to noncommutative reducts.

Section~\ref{sec:preliminaries} recalls the identity criterion for
$S_7$. Sections~\ref{sec:construction} and~\ref{sec:local} construct
the hypergraph semirings and prove the local membership result.
Section~\ref{sec:random} establishes the required hypergraph
extension lemma. Section~\ref{sec:languages} proves the regular
language statements used in Section~\ref{sec:separation} to obtain
the separating identities. Theorem~\ref{thm:nilinterval} is then
proved in Section~\ref{sec:conclusion}.
Section~\ref{sec:extensions} gives uniform consequences for finite
ai-semirings of bounded order and for nilpotency indices, and explains
the distinction from inherent nonfinite basability.
Section~\ref{sec:constants} treats signatures with distinguished
constants.

\section{Preliminaries}\label{sec:preliminaries}

\begingroup
\emergencystretch=1em
We use the standard correspondence between varieties and identities,
and Birkhoff's characterization of the variety generated by an algebra;
see Burris and Sankappanavar~\cite{BS}. All algebras and subalgebras are
nonempty. Unless a signature with constants is explicitly specified,
a generated subsemiring is formed using the two binary operations only.
For a variety $\mathcal W$, write $\Id_n(\mathcal W)$ for the identities
of $\mathcal W$ using at most $n$ distinct variables.
\par\endgroup

We shall use the following criterion for infinite axiomatic rank.

\begin{lemma}\label{lem:localglobal}
Let $\mathcal U\subseteq\mathcal W$ be varieties of ai-semirings.
Suppose that for each positive integer $n$ there is an ai-semiring $B_n$
outside $\mathcal W$ such that every subsemiring of $B_n$ generated by
at most $n$ elements belongs to $\mathcal U$.
Then $B_n$ satisfies $\Id_n(\mathcal W)$, and $\mathcal W$ has infinite
axiomatic rank.
\end{lemma}

\begin{proof}
Let $u\approx v$ belong to $\Id_n(\mathcal W)$, and assign its variables
arbitrarily to elements of $B_n$. These values generate a subsemiring
$C$ belonging to $\mathcal U\subseteq\mathcal W$. Since $u\approx v$
holds in $C$, it holds under the given assignment. Thus $B_n$ satisfies
every identity in $\Id_n(\mathcal W)$.
If these identities implied all identities of $\mathcal W$, Birkhoff's
theorem would give $B_n\in\mathcal W$, a contradiction. This applies
to every positive integer $n$.
\end{proof}

\subsection{Power semirings}

For a semigroup $M$, let $\Pow(M)$ be the algebra of its nonempty
subsets, with union as addition and setwise multiplication
\[
 XY=\{xy:x\in X,\ y\in Y\}.
\]
Associativity of the product follows from associativity in $M$;
the distributive laws follow directly from the definition. Thus
$\Pow(M)$ is an ai-semiring. The full power semiring $\mathcal P(M)$
also includes the empty set, which is an additive identity and a
multiplicative zero. We will use both constructions in
Section~\ref{sec:constants}.

\begin{proposition}\label{prop:powersum}
Let $A$ be a finite ai-semiring with multiplicative reduct $M$.
The map
\[
 \pi:\Pow(M)\longrightarrow A,\qquad
 \pi(X)=\sum_{x\in X}x,
\]
is an onto homomorphism. In particular, every identity of $\Pow(M)$
holds in $A$.
\end{proposition}

\begin{proof}
The sum defining $\pi$ is independent of the order of its summands.
Additive idempotence gives $\pi(X\cup Y)=\pi(X)+\pi(Y)$. By the two
distributive laws,
\[
 \pi(XY)=\sum_{x\in X,\,y\in Y}xy
 =\left(\sum_{x\in X}x\right)\left(\sum_{y\in Y}y\right)
 =\pi(X)\pi(Y).
\]
If different pairs give the same product, repeated copies of that
product are removed by additive idempotence. Finally,
$\pi(\{a\})=a$ for each $a\in A$, so $\pi$ is onto.
\end{proof}

\subsection{Identities of \texorpdfstring{$S_7$}{S7}}

By distributivity, every semiring term can be written as a finite
nonempty sum of nonempty words. We call such an expression a polynomial.
For a polynomial $u$, let $\cont(u)$ be the set of variables occurring
in $u$, and put
\[
 \delta(u)=\left\{Z\subseteq\cont(u):
      \begin{array}{l}
      \text{every word in }u\text{ has exactly one}\\
      \text{occurrence of a variable from }Z
      \end{array}\right\}.
\]
We refer to the members of $\delta(u)$ as transversals of $u$.
Occurrences are counted with multiplicity: if a variable occurs twice
in a word, a transversal containing that variable cannot meet that word
exactly once. The order of letters in a word does not affect either
$\cont(u)$ or $\delta(u)$.

The next characterization is due to Jackson, Ren and
Zhao~\cite[Proposition~5.5]{JRZ}.

\begin{proposition}\label{prop:criterion}
For semiring polynomials $u$ and $v$,
\[
 S_7\models u\approx v
 \quad\Longleftrightarrow\quad
 \cont(u)=\cont(v)\quad\text{and}\quad\delta(u)=\delta(v).
\]
\end{proposition}

\section{Hypergraph semirings}\label{sec:construction}

Flat algebras provide a way to relate universal Horn theories to
equational theories; see Jackson~\cite{Jflat}.
Hypergraphs have been used to study both the membership problem and
axiomatisability for classes of finite algebras;
see Kun and V\'ertesi~\cite{KV}, Ham and Jackson~\cite{HJ}, and,
for ai-semirings, Jackson, Ren and Zhao~\cite{JRZ}.
In the construction below, both the size of an edge and the permitted
size of an intersection are parameters. Their relation will allow us
to control subsemirings generated by a prescribed number of elements.

Let $d\ge1$ and $k\ge2$ be integers such that
\begin{equation}\label{eq:degreebound}
 k>2(d-1).
\end{equation}
Let $\HH=(V,\mathcal E)$ be a finite, nonempty $k$-uniform hypergraph,
with $V=\bigcup\mathcal E$, and suppose that
\begin{equation}\label{eq:intersection}
 |E\cap F|<d\qquad(E,F\in\mathcal E,\ E\ne F).
\end{equation}
We call the members of $\mathcal E$ blocks.

The nonzero elements of a semigroup $F(\HH)$ are the nonempty subsets
of blocks. Their product is their union if they are disjoint and their
union lies in a block; otherwise it is $0$. The element $0$ is absorbing.
This operation is associative: a product of several nonzero factors is
their union precisely when the factors are pairwise disjoint and their
union lies in a block. If either condition fails for a partial product,
it also fails after further factors are adjoined.

In this semigroup, identify all full blocks. Also, for each nonempty
$R$ contained in a block with $|R|\le d-1$, identify all the sets
$E\setminus R$ with $E\supseteq R$. The resulting nonzero elements are
denoted by
\[
\begin{array}{ll}
 a_U,&\varnothing\ne U\subseteq E\in\mathcal E,\quad |U|\le k-d,\\[2pt]
 b_R,&\varnothing\ne R\subseteq E\in\mathcal E,\quad |R|\le d-1,\\[2pt]
 p,&\text{the common class of the blocks.}
\end{array}
\]
Assign these elements degrees $|U|$, $k-|R|$, and $k$, respectively,
and write $b_{\varnothing}=p$ where convenient. The degree ranges of
the first two types are disjoint. In particular, when $d=1$, the blocks
are pairwise disjoint and there are no elements $b_R$ with $R\ne\varnothing$.
After taking this multiplicative quotient, define addition by
\begin{equation}\label{eq:flataddition}
 x+x=x,\qquad x+y=0\quad(x\ne y).
\end{equation}
We denote the resulting algebra by $B(\HH)$.

Recall that a semigroup with zero is \emph{$0$-cancellative} if
$xy=xz\ne0$ implies $y=z$, and $yx=zx\ne0$ implies $y=z$.

\begin{lemma}\label{lem:semiring}
The stated identifications form a semigroup congruence. The quotient is
$0$-cancellative, and $B(\HH)$ is a finite commutative ai-semiring.
\end{lemma}

\begin{proof}
Any set of size at least $k-d+1$ that is contained in a block is
contained in exactly one block, since $k-d+1>d-1$.
Thus a subset in one of the prescribed complement classes has a
unique host block and a unique complement label. In particular, two
such classes cannot overlap unless their labels agree. The full-block
class has degree $k$ and the ordinary singleton classes have degree
at most $k-d$, so these are disjoint from the complement classes as
well. The zero class is the singleton $\{0\}$.

If a product
$(E\setminus R)U$ is nonzero in $F(\HH)$, its containing block must be
$E$. Hence $U\subseteq R$, and its class is $b_{R\setminus U}$.
Conversely, this product is nonzero whenever $U\subseteq R$.
The output class is therefore independent of the representative
$E\setminus R$. Full blocks annihilate every nonempty factor.
These observations prove compatibility with multiplication by an
arbitrary element of $F(\HH)$ on either side, and hence prove that
the stated classes form a semigroup congruence. In particular, a
product of ordinary representatives is zero in $F(\HH)$ if and only
if its image is zero in the quotient. We have formed a quotient of
the multiplicative semigroup; flat addition is imposed on that
quotient afterwards.

For reference, the quotient multiplication is as follows. The product
$a_Ua_W$ is zero unless $U,W$ are disjoint and $L=U\cup W$ lies in a
block. Otherwise it is $a_L$ if $|L|\le k-d$, $b_{E\setminus L}$ if
$k-d<|L|<k$, and $p$ if $|L|=k$. In the middle case, $E$ is unique.
Moreover,
\begin{equation}\label{eq:complementproduct}
 b_Ra_U=
 \begin{cases}
 b_{R\setminus U},&U\subseteq R,\\
 0,&U\not\subseteq R.
 \end{cases}
\end{equation}
Two elements of type $b_R$ multiply to zero, because their degrees sum
to more than $k$, and $p$ annihilates every element. Degrees add whenever
a product is nonzero.

To prove $0$-cancellation, fix a left factor and a nonzero output. We
show that there is at most one right factor.

If the left factor is $a_U$ with $|U|\ge d$, its containing block $E$
is unique. The right factor must be ordinary, since a factor $b_R$
would make the total degree exceed $k$. Outputs $a_L$, $b_R$, and $p$
respectively determine the right supports
\[
 L\setminus U,\qquad E\setminus(R\cup U),\qquad E\setminus U.
\]
This proves uniqueness whenever such a product exists.

Suppose $1\le |U|\le d-1$. An output $a_L$ determines
$a_{L\setminus U}$, and output $p$ determines $b_U$. For output $b_R$,
one must have $U\cap R=\varnothing$. If $|U\cup R|\ge d$, there is at
most one block containing $U\cup R$, and it determines the right
factor $a_{E\setminus(U\cup R)}$. If $|U\cup R|<d$, the right factor
is necessarily $b_{U\cup R}$. Indeed, all possible representing
complements have this label, whereas an ordinary right factor would
have degree $k-|U\cup R|>k-d$.

Finally, a left factor $b_R$ and output $b_Q$ determine
$a_{R\setminus Q}$; output $p$ determines $a_R$. There are no nonzero
products with left factor $p$. Empty or inadmissible supports in these
formulas simply give no product. We have proved
$xy=xz\ne0\Rightarrow y=z$.

Flat addition is associative, commutative, and idempotent. If $y\ne z$,
then $xy,xz$ cannot be equal and nonzero, so
$xy+xz=0=x(y+z)$. If $y=z$, distributivity follows from idempotence.
Commutativity of multiplication gives the other distributive law.
\end{proof}

Every element of $B(\HH)$ has square zero. Also, any product of nonzero
elements whose degrees sum to more than $k$ is zero.

\subsection{Quotients determined by divisibility}

For $r\ne0$, an element $x$ divides $r$ if $xy=r$ for some
$y\in B(\HH)^1$, where the superscript allows an identity to be
adjoined for this definition. The nondivisors form a multiplicative
ideal: if $xz$ divides $r$, then $x$ divides $r$. They contain $0$
and are closed under addition with any element, by flatness.
Consequently they may be collapsed to zero to give a semiring Rees
quotient. Every retained nonzero element has a singleton preimage.

\begin{lemma}\label{lem:link}
Let $r=b_R$, with $1\le s=|R|\le d-1$. The Rees quotient obtained by
collapsing its nondivisors is isomorphic to $B(\HH_R)$, where
\[
 \HH_R=\{E\setminus R:E\in\mathcal E,\ R\subseteq E\}.
\]
The vertex set is the union of these blocks, the parameters are
$k'=k-s$ and $d'=d-s$, and $b_R$ corresponds to the common block product.
\end{lemma}

\begin{proof}
An ordinary element $a_U$ divides $b_R$ precisely when $U\cap R$ is
empty and $U\cup R$ lies in a block. Necessity follows from the product
formulas above. For sufficiency, choose a block $E\supseteq U\cup R$
and multiply $a_U$ by the class of $E\setminus(U\cup R)$. This
complement is nonempty, since $|U|+|R|\le k-1$.
These divisors are exactly the ordinary elements of $B(\HH_R)$,
because $k'-d'=k-d$.

The element $b_Q$ divides $b_R$ precisely when $R\subseteq Q$.
For $Q\ne R$ it corresponds to the link element $b_{Q\setminus R}$;
the element $b_R$ itself corresponds to the common block product.
The element $p$ does not divide $b_R$. This gives a bijection between
the nonzero elements of the two algebras.

Formula~\eqref{eq:complementproduct} is preserved by this bijection:
for a retained $b_Q$ and an ordinary divisor $a_U$, the product
$b_Qa_U=b_{Q\setminus U}$ survives precisely when
$U\subseteq Q$ and $R\subseteq Q\setminus U$. Since
$U\cap R=\varnothing$, this is equivalent to
$U\subseteq Q\setminus R$, the corresponding condition in the link.

Consider next two retained ordinary elements $a_U,a_W$. If their
supports overlap, their product is zero in both constructions.
Suppose that the supports are disjoint and put $L=U\cup W$.
When $|L|\le k-d$, their product survives precisely when $L$ is a
divisor support, that is, when $L\cap R=\varnothing$ and $L\cup R$
lies in a block. When $|L|>k-d$, a nonzero original product has a
unique containing block $E$ and is labelled by $E\setminus L$,
with the empty label allowed for $p$. It survives precisely when
$R\subseteq E\setminus L$. This is again equivalent to saying that
$L$ is disjoint from $R$ and is contained in a block containing $R$.
In the surviving case its link label is
$(E\setminus L)\setminus R=(E\setminus R)\setminus L$, as required;
the label is empty precisely when $L=E\setminus R$.
These are exactly the nonzero-product conditions and outputs in
$B(\HH_R)$. Products of two large elements are zero in both algebras.
The bijection therefore preserves multiplication. It preserves flat
addition as well.

Distinct link blocks intersect in fewer than $d-s$ points. Finally,
\[
 k'-2(d'-1)=k-2(d-1)+s>0.
\]
Here $d'\ge1$ and $k'\ge k-d+1\ge d\ge2$, so the construction applies,
including the case $d'=1$.
\end{proof}

\section{Finitely generated subsemirings}\label{sec:local}

\begin{theorem}\label{thm:local}
Let $n\ge1$. Under the assumptions of the preceding section, if
\begin{equation}\label{eq:localbound}
 k>n(d-1),
\end{equation}
then every subsemiring of $B(\HH)$ generated by at most $n$ elements
belongs to $\mathcal N$.
\end{theorem}

We first consider polynomials all of whose monomials have value $p$.

\begin{lemma}\label{lem:topfibre}
Assign distinct variables to at most $n$ distinct nonzero elements of
$B(\HH)$, under the hypotheses of Theorem~\ref{thm:local}. Suppose
every monomial of $u$ has value $p$. Let $w$ be a word using only
variables from $\cont(u)$. If every $Z\in\delta(u)$ meets $w$ in
exactly one occurrence, then $w$ has value $p$.
\end{lemma}

\begin{proof}
Retain only generators whose variables occur in $u$.

\smallskip\noindent
\textbf{Step 1.} The possible monomials and their transversals.
Call an ordinary generator $a_U$ small if $|U|\le d-1$; otherwise
$|U|\ge d$, and it has a unique containing block $E(U)$.
Let $D_0$ be the union of the small supports. Then
$|D_0|\le n(d-1)<k$.

For a word consisting of ordinary factors, take their actual supports
as representatives in $F(\HH)$. The quotient map is a semigroup
homomorphism and its zero class is a singleton. The ordinary word
therefore has nonzero value in the quotient precisely when its
representatives have nonzero product in $F(\HH)$. Since the preimage
of $p$ consists exactly of the full blocks, such a word has value $p$
precisely when its supports are pairwise disjoint and their union
is a block. A nonzero word involving $b_R$ has only
ordinary remaining factors, and~\eqref{eq:complementproduct} shows that
their supports must partition $R$ in order to give $p$.
A word involving the factor $p$ has value $p$ only if it has no other
factor. We have therefore obtained exactly three possibilities:
$p$ alone, a factor $b_R$ followed by small factors whose supports
partition $R$, or ordinary factors whose supports partition a block
$E$. In the second case $R\subseteq D_0$. In the third case at least
one factor is not small, since $E\not\subseteq D_0$.
Call the blocks occurring in the third case active. Every nonsmall
ordinary generator occurring in $u$ belongs to a unique active block.

Choose independently $c_E\in E\setminus D_0$ for each active block.
Select the variables corresponding to $p$ and all $b_R$, select no
small variable, and select a nonsmall ordinary variable $a_U$ exactly
when $c_{E(U)}\in U$. Each monomial of $u$ has exactly one selected
occurrence: in the first two forms it is the distinguished factor;
in the third it is the unique part containing $c_E$. Thus these
selections belong to $\delta(u)$.

There is a second selection for each $x\in D_0$. Select $p$, select
$b_R$ when $x\notin R$, and select a small $a_U$ when $x\in U$.
For each active block containing $x$, set $c_E=x$; for every other
active block choose $c_E\in E\setminus D_0$. Select the nonsmall
ordinary variables by membership of $c_E$ in their support.
For a word involving $b_R$, either that factor is selected or exactly
one part of its partition of $R$ is selected. For an ordinary block
partition, exactly one part contains $c_E$; if $x\notin E$, none of
its small parts is selected. Hence these selections also lie in
$\delta(u)$.

\smallskip\noindent
\textbf{Step 2.} Two counting equations.
Let $m_g$ be the number of occurrences in $w$ of the variable naming
generator $g$; write $m_p=0$ if $p$ is absent. All sums over $R$ below
range over complement generators occurring in
$u$, and all sums over $E$ range over active blocks. For an active block $E$,
put
\[
 f_E(x)=\sum_{\substack{g=a_U,\ |U|\ge d\\E(U)=E,\ x\in U}}m_g.
\]
The first selections give
\[
 m_p+\sum_Rm_{b_R}+\sum_Ef_E(c_E)=1.
\]
Fix an active block $E$ and hold the choices for all other active
blocks fixed. Each point of $E\setminus D_0$ can be used as $c_E$,
and each resulting selection is a transversal of $u$. Comparing the
displayed equality for two such choices shows that $f_E$ takes the
same value at those points. Thus $f_E$ is constant on
$E\setminus D_0$. Denote its constant value by $d_E\in\NN$. The
argument applies to every active block separately, giving
\begin{equation}\label{eq:base}
 m_p+\sum_Rm_{b_R}+\sum_Ed_E=1.
\end{equation}
For $x\in D_0$, let
\[
 C_w(x)=\sum_{\substack{g=a_U\\x\in U}}m_g.
\]
The second selection at $x$ gives
\[
 m_p+\sum_{R\not\ni x}m_{b_R}+C_w(x)
       +\sum_{E\not\ni x}d_E=1.
\]
Indeed, all ordinary factors whose supports contain $x$ are selected;
each active block missing $x$ contributes its constant $d_E$ instead.
Subtracting~\eqref{eq:base} yields
\begin{equation}\label{eq:point}
 C_w(x)=\sum_{R\ni x}m_{b_R}+\sum_{E\ni x}d_E
 \qquad(x\in D_0).
\end{equation}

\smallskip\noindent
\textbf{Step 3.} Determine the value of $w$.
The nonnegative integers in~\eqref{eq:base} sum to one, so exactly
one of the following cases occurs.

If $m_p=1$, all $m_{b_R}$ and $d_E$ vanish. Every nonsmall ordinary
support occurring in $w$ must lie in $D_0$, since its block contributes
zero outside $D_0$. But~\eqref{eq:point} gives $C_w(x)=0$ throughout
$D_0$, excluding all ordinary factors. Thus $w=p$.

If $m_{b_R}=1$ for some $R$, all other terms in~\eqref{eq:base} vanish.
Again every ordinary support occurring in $w$ lies in $D_0$, and
\eqref{eq:point} now gives $C_w(x)=\1_R(x)$ there. Therefore the
ordinary supports, counted with multiplicity, partition $R$. They
are all small because $|R|\le d-1$. Together with $b_R$ they give
value $p$.

Finally, suppose $d_{E_0}=1$ for an active block $E_0$. There are no
$p$ or $b_R$ factors. If a nonsmall factor belongs to another block
$F$, the equation $d_F=0$ confines its support to $D_0$, while
\eqref{eq:point} confines it to $E_0$. Its support would then have
size at most $|E_0\cap F|<d$, a contradiction. Thus all nonsmall
factors belong to $E_0$. Their supports cover each point of
$E_0\setminus D_0$ exactly once, since $f_{E_0}=1$ there.
Equation~\eqref{eq:point} says that all ordinary supports cover
precisely $E_0\cap D_0$ once inside $D_0$. Hence the supports partition
$E_0$, and $w$ again has value $p$.

All counts include repeated occurrences of variables. The argument
also applies when $D_0$ is empty: then the second family of equations
is empty, and the same three cases give the conclusion.
\end{proof}

\begin{proof}[Proof of Theorem~\ref{thm:local}]
Let $C$ be generated by a set of at most $n$ distinct elements, and
consider an identity of $S_7$ under an assignment into $C$. Express
each assigned element as a polynomial in distinct symbols naming
the generators. Substitution gives an identity $u\approx v$ of
$S_7$, and Proposition~\ref{prop:criterion} gives
\[
 \cont(u)=\cont(v),\qquad \delta(u)=\delta(v).
\]
If both evaluations are zero, they agree. Otherwise interchange the
sides if necessary, and let $u$ have nonzero value $r$. By flatness,
every monomial of $u$ has value $r$. Every generator occurring in $u$
is therefore nonzero; discard the other generators, which also do
not occur in $v$. Each monomial $w$ of $v$ uses only $\cont(u)$ and
is met exactly once by every member of $\delta(u)$.

If $r=p$, Lemma~\ref{lem:topfibre} gives $w=p$.
If $r=a_W$, all occurring generators are ordinary, and every
monomial of $u$ partitions $W$ by their supports. For each $x\in W$,
the generator symbols whose supports contain $x$ form a transversal
of $u$. Therefore the supports in $w$ cover every point of $W$
exactly once. They are all contained in $W$, so they partition $W$
and $w=a_W$.

Suppose $r=b_R$, and put $s=|R|$. Every occurring generator divides
$r$, since it is a factor of a monomial with value $r$. Pass to the
Rees quotient of Lemma~\ref{lem:link}. These generators remain
distinct and nonzero, and $r$ becomes the common block product.
The new parameters satisfy
\begin{align*}
 (k-s)-2(d-s-1)&=k-2(d-1)+s>0,\\
 (k-s)-n(d-s-1)&=k-n(d-1)+(n-1)s>0.
\end{align*}
Lemma~\ref{lem:topfibre} applies there and gives $w=b_R$ in the
quotient. Since a retained nonzero element has a singleton preimage,
$w=b_R$ in $B(\HH)$ as well.

In all cases every monomial of $v$ has value $r$, so $u=v$.
Consequently $C$ satisfies every identity of $S_7$ and belongs to
$\Var(S_7)$ by Birkhoff's theorem. Every element of $B(\HH)$ has
square $0$, and $0$ absorbs multiplication. Thus $C$ also satisfies
$x^2y\approx x^2$, and hence $C\in\mathcal N$. Notice that $0$
belongs to every nonempty subsemiring $C$, since it is the square of
any element of $C$. The generators considered throughout the proof
were arbitrary elements of $B(\HH)$; no assumption that they are
vertex elements has been made.
\end{proof}

\section{Hypergraphs with prescribed extensions}\label{sec:random}

The use of sparse hypergraphs with prescribed global properties goes
back to the constructions of Erd\H{o}s and Hajnal~\cite{EH}.
For the present argument we need an extension property that holds
simultaneously for every colouring and every set of bounded size.
The following probabilistic construction supplies this property while
retaining a bound on intersections.

\begin{lemma}\label{lem:random}
Let $m,D$ be positive integers, let $\varepsilon>0$, and put $d=2D+4$.
For every fixed integer $k>d$ and every sufficiently large integer $N$,
there is a $k$-uniform hypergraph $\HH=(V,\mathcal E)$ with $|V|=N$
having the following properties.
\begin{enumerate}[$(1)$]
\item If $E,F\in\mathcal E$ are distinct, then $|E\cap F|<d$.
\item Let $\kappa:V\to\{1,\ldots,m\}$ be any colouring, let
      $R\subseteq V$ with $|R|\le D$, and let $q\in\NN^m$ satisfy
      $\len q=k-|R|$. If
      \[
       q_i>0\quad\Longrightarrow\quad
       |\kappa^{-1}(i)\setminus R|\ge\varepsilon N,
      \]
      then there is an $E\in\mathcal E$ containing $R$ such that
      \[
       |(E\setminus R)\cap\kappa^{-1}(i)|=q_i
       \qquad(1\le i\le m).
      \]
\end{enumerate}
In particular, $N$ may be chosen sufficiently large in any prescribed
residue class modulo a fixed positive integer.
\end{lemma}

\begin{proof}
Fix a set $V$ of size $N$ and select each of its $k$-element subsets
independently with probability
\[
 \rho=N^{-k+D+3}.
\]
Since $k>2D+4$, we have $0<\rho\le1$ for $N\ge1$.

Call a triple $(\kappa,R,q)$ \emph{eligible} if it satisfies the
hypotheses in (2), and put $r=|R|$. The number of $k$-element subsets
that would provide the required extension is
\[
 T(\kappa,R,q)=
 \prod_{i=1}^m\binom{|\kappa^{-1}(i)\setminus R|}{q_i}.
\]
There is a constant $c>0$, depending only on the fixed parameters
$m,D,k,\varepsilon$, such that
\[
 T(\kappa,R,q)\ge cN^{k-r}
\]
for every eligible triple and all sufficiently large $N$. Indeed, if
$q_i>0$ and $N\ge2k/\varepsilon$, then
\[
 \binom{|\kappa^{-1}(i)\setminus R|}{q_i}
 \ge \frac{(\varepsilon N/2)^{q_i}}{q_i!}.
\]
There are only finitely many possible vectors $q$ with $\len q\le k$,
so the constants in the resulting product have a common positive lower
bound. The same $c$ therefore works for every eligible colouring,
every core $R$ of size at most $D$, and every permitted vector $q$.
All parameters $m,D,k,\varepsilon$ are fixed before $N$ tends to
infinity.

Let $X$ be the number of selected extensions of a fixed eligible triple.
It is binomial, with mean
\[
 \mu=T(\kappa,R,q)\rho
      \ge cN^{D+3-r}\ge cN^3.
\]
We use the following elementary Chernoff bound~\cite{Chernoff}:
\[
 \Pr(X\le\mu/2)\le e^{-\mu/8}.
\]
For completeness, it follows by applying Markov's inequality to
$2^{-X}$:
\[
 \Pr(X\le\mu/2)
 \le 2^{\mu/2}\mathbb E(2^{-X})
 =2^{\mu/2}(1-\rho/2)^{T(\kappa,R,q)}
 \le \exp\bigl((\log 2-1)\mu/2\bigr)
 \le e^{-\mu/8}.
\]
Consequently, the probability that this triple has fewer than
$(c/2)N^3$ selected extensions is at most $e^{-cN^3/8}$.
There are at most
\[
 m^N\left(\sum_{r=0}^D\binom Nr\right)(k+1)^m
\]
triples. Multiplying this bound by $e^{-cN^3/8}$ gives a quantity tending
to zero. Thus, with probability tending to one, every eligible triple
has at least $(c/2)N^3$ selected extensions.

It remains to control intersections. Let $Y$ count ordered pairs of
distinct selected blocks with intersection of size at least $d$. Then
\begin{align*}
 \mathbb EY
 &\le\binom Nk\sum_{j=d}^{k-1}
       \binom kj\binom{N-k}{k-j}\rho^2\\
 &=O(N^{2k-d}\rho^2)
  =O(N^{2D+6-d})=O(N^2).
\end{align*}
The constants here depend only on $k$ and $D$. Hence
$\Pr(Y>N^{5/2})=O(N^{-1/2})$ by Markov's inequality.
For sufficiently large $N$, choose a realization in which every eligible
triple has at least $(c/2)N^3$ extensions and $Y\le N^{5/2}$.
Delete every block belonging to a pair counted by $Y$. This removes at
most $2N^{5/2}$ blocks. Since
\[
 (c/2)N^3>2N^{5/2}
\]
for all sufficiently large $N$, at least one extension of every eligible
triple remains. The surviving blocks satisfy both (1) and (2).
The argument applies to every sufficiently large integer $N$, which also
proves the final assertion.
\end{proof}

\section{Multiplicity vectors of words in a finite semigroup}
\label{sec:languages}

For an alphabet of $m$ letters, the \emph{multiplicity vector} of a word
is the vector in $\NN^m$ whose $i$th coordinate counts the occurrences
of the $i$th letter. Its coordinate sum is the length of the word.
Such vectors are usually called \emph{Parikh vectors}. Parikh's
theorem establishes semilinearity for context-free languages
\cite{Parikh}; the regular-language case and its quantitative aspects
are discussed by To~\cite{To}. We need only the regular-language case,
for which we give a direct proof below.

A subset of $\NN^m$ is \emph{linear} if it has the form
\[
 b+\NN P
 =\left\{b+\sum_{h\in P}\lambda_h h:
                   \lambda_h\in\NN\right\},
\]
where $b\in\NN^m$ and $P\subseteq\NN^m$ is finite. It is
\emph{semilinear} if it is a finite union of linear sets. Empty unions
and empty sets of periods are allowed.

\begin{lemma}\label{lem:semilinear}
The set of multiplicity vectors of a regular language is semilinear.
\end{lemma}

\begin{proof}
Semilinear sets are closed under finite unions and Minkowski sums, since
\[
 (b+\NN P)+(c+\NN Q)=(b+c)+\NN(P\cup Q).
\]
They are also closed under forming the additive submonoid generated by
a set. To see this, write
$U=\bigcup_{i=1}^r(b_i+\NN P_i)$ and let $U^{\oplus *}$ be the set
of all finite sums of elements of $U$, including the empty sum. Then
\begin{equation}\label{eq:star}
 U^{\oplus *}=\{0\}\ \cup\!
 \bigcup_{\varnothing\ne I\subseteq\{1,\ldots,r\}}
 \left(\sum_{i\in I}b_i+
 \NN\left(\{b_i:i\in I\}\cup\bigcup_{i\in I}P_i\right)\right).
\end{equation}
For the forward inclusion, choose a linear component for each summand
and let $I$ consist of the components used. One copy of each $b_i$,
$i\in I$, contributes the displayed base; further copies of bases and
all periods contribute to the generated monoid. Conversely, an element
on the right is a sum of elements of $U$: use one summand from each
component in $I$, place all of that component's period contributions in
this summand, and add the required further copies of its base.

\begingroup
\emergencystretch=3em
Multiplicity vectors take language union to set union, concatenation
to Minkowski sum, and Kleene star to additive star. Single letters, the
empty word, and the empty language have semilinear sets of multiplicity
vectors. It follows by induction that every regular expression does too.
For completeness, a finite automaton gives such an expression as follows.
Number its states $1,\ldots,t$, and let $R_{ij}^{(r)}$ describe paths
from $i$ to $j$ whose internal states lie in $\{1,\ldots,r\}$.
The expressions $R_{ij}^{(0)}$ consist of the labels of direct edges,
together with the empty word when $i=j$. Successively use
\par\endgroup
\[
 R_{ij}^{(r)}=
 R_{ij}^{(r-1)}\ \cup\
 R_{ir}^{(r-1)}\bigl(R_{rr}^{(r-1)}\bigr)^*
 R_{rj}^{(r-1)}.
\]
After the last state, take the union over the initial and accepting
states. Thus every regular language has the required semilinear image.
\end{proof}

\begin{lemma}\label{lem:core}
Let $M$ be a finite semigroup, regarded as an alphabet of size $m$.
For $s\in M$, let $L_s$ consist of all nonempty words whose ordered
product in $M$ is $s$. There are integers $D,L\ge1$ and a real number
$\varepsilon>0$, depending only on $M$, with the following property.
For every sufficiently long word $w\in L_s$, write $z$ for its
multiplicity vector and $N$ for its length. There are vectors
$c,h\in\NN^m$, with $h\ne0$, such that
\begin{enumerate}[$(1)$]
\item $c\le z$ coordinatewise, $\len c\le D$, and
      $\len c\equiv N\pmod L$;
\item $\len h$ divides $L$, and $c+\mu h$ is the multiplicity vector
      of a word in $L_s$ for every $\mu\in\NN$;
\item $h_i>0$ implies $z_i\ge2\varepsilon N$.
\end{enumerate}
The length threshold is independent of $s$. It may also be chosen so
that, whenever $h_i>0$, one has $z_i-c_i\ge\varepsilon N$.
\end{lemma}

\begin{proof}
Each language $L_s$ is regular. An automaton recognizing it has the states
of $M$ and a new initial state. Reading a letter $a$ in the initial state
leads to $a$; reading $a$ in state $x\in M$ leads to $xa$. The only
accepting state is $s$, and the initial state is not accepting.

By Lemma~\ref{lem:semilinear}, fix a semilinear representation for the
multiplicity vectors of each $L_s$. There are finitely many linear
components in these representations altogether. Delete zero period
vectors. At least one nonzero period remains, since the union of the
languages $L_s$ is $M^+$ and has words of arbitrarily large length.
Let $L$ be the least common multiple of the lengths of all the nonzero
period vectors. Put
\[
 D=\max\left\{1,\
 \max_{(b,P)}\left(\len b+(L-1)\sum_{h\in P}\len h\right)\right\},
\]
where $(b,P)$ ranges over all the chosen components. Let $B_0$ be the
largest base length, let $J\ge1$ be the largest number of periods in a
component, and let $C\ge1$ be the largest period length. Set
\[
 \varepsilon=\frac1{4JC}.
\]

Take $w\in L_s$ of length $N>2B_0$ and choose a component containing
its multiplicity vector. Write
\[
 z=b+\sum_{i=1}^j\lambda_i v_i.
\]
Since $N>2B_0$, this component has at least one period. Moreover,
\[
 \sum_{i=1}^j\lambda_i\len{v_i}
 =N-\len b\ge N/2.
\]
As $j\le J$ and $\len{v_i}\le C$, some index $a$ satisfies
\[
 \lambda_a\ge\frac{N}{2JC}=2\varepsilon N.
\]
Set $h=v_a$ and
\[
 c=b+\sum_{i=1}^j(\lambda_i\bmod L)v_i,
\]
where each remainder belongs to $\{0,\ldots,L-1\}$.
Then $c\le z$ and $\len c\le D$. Every coordinate of $z-c$ is
divisible by $L$, so $\len c\equiv N\pmod L$. Also $\len h$ divides
$L$ by the choice of $L$. Adding any nonnegative multiple of $h$ to
$c$ keeps the vector in the same linear component, proving (2).
If $h_i>0$, its $i$th coordinate contributes at least $\lambda_a$ to
$z_i$, which proves (3). Finally, $c_i\le D$ gives
\[
 z_i-c_i\ge2\varepsilon N-D\ge\varepsilon N
\]
when $N\ge D/\varepsilon$. All constants were chosen using all $s\in M$
at once, so a common length threshold suffices.
\end{proof}

\section{An identity separating the hypergraph semiring}
\label{sec:separation}

For a hypergraph $\HH=(V,\mathcal E)$, associate a variable $x_v$ with
each vertex. Choose an ordering of $V$ and define
\begin{align}
 q_{\HH}&=\prod_{v\in V}x_v
                  &&\text{in the chosen order},\label{eq:globalword}\\
 t_{\HH}&=\sum_{E\in\mathcal E}
          \sum_{(v_1,\ldots,v_k)\in\operatorname{Ord}(E)}
          x_{v_1}\cdots x_{v_k},\label{eq:edgeterm}
\end{align}
where $\operatorname{Ord}(E)$ denotes all orderings of the vertices of
$E$. The term $t_{\HH}$ includes all these orderings because the
semigroup to which the identity will be applied need not be commutative.

\begin{proposition}\label{prop:separation}
Let $M$ be a finite semigroup and let $n\ge1$. There exist integers
$d,k,N$ and a hypergraph $\HH$ satisfying all the hypotheses of
Theorem~\ref{thm:local}, with $|V|=N>k$, for which
\begin{equation}\label{eq:separator}
 \Pow(M)\models t_{\HH}\approx t_{\HH}+q_{\HH},
 \qquad
 B(\HH)\not\models t_{\HH}\approx t_{\HH}+q_{\HH}.
\end{equation}
Here $\Pow(M)$ is the ai-semiring of nonempty subsets of $M$, with
union as addition and setwise multiplication.
\end{proposition}

\begin{proof}
Choose $D,L,\varepsilon$ for $M$ as in Lemma~\ref{lem:core} and put
$d=2D+4$. Choose an integer
\[
 k>\max\{d,\ 2(d-1),\ n(d-1),\ D\}.
\]
With these parameters fixed, choose $N>k$ sufficiently large for
Lemmas~\ref{lem:random} and~\ref{lem:core} to apply and such that
$N\equiv k\pmod L$. Lemma~\ref{lem:random}, with $m=|M|$, provides
a hypergraph $\HH$ on $N$ vertices satisfying the stated intersection
and extension properties. It has no isolated vertices. Indeed, given
$v\in V$, take $R=\{v\}$, colour every vertex with one colour, and
request $k-1$ further vertices of that colour. This request is eligible
for large $N$, since $D\ge1$ and $\varepsilon\le1/4$, and so a block
containing $v$ exists. The choices of $d,k$ also give
\[
 k>2(d-1),\qquad k>n(d-1),
\]
as required for Theorem~\ref{thm:local}.

We prove the first assertion of~\eqref{eq:separator}. Assign an arbitrary
nonempty subset $X_v\subseteq M$ to each variable $x_v$, and take
$s\in q_{\HH}(X_v:v\in V)$. By the definition of setwise multiplication,
there are elements $\alpha_v\in X_v$ whose product in the chosen order
of $V$ is $s$. Regard these elements, in that order, as a word $w\in L_s$,
and let $z$ be its multiplicity vector. Apply Lemma~\ref{lem:core} to
obtain $c,h$.

Colour vertex $v$ by the element $\alpha_v\in M$. Since $c\le z$, we
can select a set $R\subseteq V$ containing exactly $c_i$ vertices of
each colour $i$. Thus $|R|=\len c\le D$, and
\[
 |R|\equiv N\equiv k\pmod L.
\]
Because $\len h$ divides $L$ and $k>D\ge|R|$, the number
\[
 \mu=\frac{k-|R|}{\len h}
\]
is a positive integer. The vector $\mu h$ has coordinate sum
$k-|R|$. Every colour used by it still occurs at least $\varepsilon N$
times outside $R$, by Lemma~\ref{lem:core}. The extension property
therefore gives a block $E\supseteq R$ whose vertices outside $R$ have
multiplicity vector $\mu h$. The colours on the whole block have
multiplicity vector $c+\mu h$.

By Lemma~\ref{lem:core}(2), there is a word in $L_s$ with precisely this
vector. For each letter $i$, the number of its occurrences in that
word equals the number of vertices of $E$ coloured $i$. Choose a
bijection between those occurrences and those vertices, independently
for each colour. Reading the matched vertices in the order of the
word lists each vertex of $E$ exactly once. This yields an ordering
$(v_1,\ldots,v_k)$ of $E$ for which
\[
 \alpha_{v_1}\cdots\alpha_{v_k}=s.
\]
The ordered product here is the product of the chosen word in $L_s$;
no commutation of factors in $M$ is involved. The corresponding
monomial is a summand of $t_{\HH}$, and each
$\alpha_{v_i}$ belongs to $X_{v_i}$. Hence
$s\in t_{\HH}(X_v:v\in V)$. This proves
\[
 q_{\HH}(X_v:v\in V)\subseteq t_{\HH}(X_v:v\in V)
\]
for every assignment, which is exactly the asserted identity in
$\Pow(M)$.

For the second assertion, evaluate $x_v$ as $a_{\{v\}}$ in $B(\HH)$.
The product over each block, in any order, is $p$, so $t_{\HH}$ has
value $p$. The word $q_{\HH}$ is a product of $N>k$ degree-one elements
and therefore has value $0$. Its addition to $p$ gives $0$, whereas
$p\ne0$. Thus the identity fails in $B(\HH)$.
\end{proof}

\section{Proof of the main theorem}
\label{sec:conclusion}

\begin{proof}[Proof of Theorem~\ref{thm:nilinterval}]
Let $A$ be a finite ai-semiring and suppose that
$\mathcal N\subseteq\mathcal W\subseteq\Var(A)$.
Fix a positive integer $n$, and let $M$ be the multiplicative reduct
of $A$. Apply Proposition~\ref{prop:separation} to $M$ and $n$, and
write $B_n=B(\HH)$ for the resulting finite semiring. The identity
\[
 t_{\HH}\approx t_{\HH}+q_{\HH}
\]
holds in $\Pow(M)$ and hence, by Proposition~\ref{prop:powersum}, in
$A$. It therefore holds throughout $\mathcal W$. Since it fails in
$B_n$, we have $B_n\notin\Var(A)$ and therefore
$B_n\notin\mathcal W$.

Now let $u\approx v$ be any identity of $\mathcal W$ involving at most
$n$ variables. Under an arbitrary assignment of these variables in
$B_n$, their values generate a subsemiring $C$.
Theorem~\ref{thm:local} gives $C\in\mathcal N\subseteq\mathcal W$.
Consequently the identity holds under the assignment. Thus $B_n$
satisfies every identity of $\mathcal W$ in at most $n$ variables.
Since $B_n\notin\mathcal W$, these identities cannot form a basis
for $\mathcal W$. This holds for every $n$, proving infinite
axiomatic rank. The finite, commutative, flat and nilpotent properties
of $B_n$, and the equality of all its squares to its absorbing top,
follow from its construction.

Finally, if $S_7\in\Var(A)$, then
$\mathcal N\subseteq\Var(S_7)\subseteq\Var(A)$. Taking
$\mathcal W=\Var(A)$ shows that $A$ is nonfinitely based.
Since $A$ was arbitrary, $S_7$ is strongly nonfinitely based, as
asserted in Corollary~\ref{thm:main}.
\end{proof}

\section{Further consequences}
\label{sec:extensions}

\subsection{Equational definitions on finite nilpotent semirings}

Let $\mathcal F_{\mathrm{nil}}$ denote the class of finite commutative
flat semirings whose multiplicative reducts are nilpotent and in which
all elements have square equal to the absorbing top.
The next consequence concerns identities valid only on the specified
finite members of a variety.

\begin{corollary}\label{cor:finiteequational}
Suppose that $A$ is finite and
$\mathcal N\subseteq\mathcal W\subseteq\Var(A)$.
There is no set $\Sigma$ of identities with a fixed finite bound on
the number of variables such that, for every
$B\in\mathcal F_{\mathrm{nil}}$,
\[
 B\models\Sigma\quad\Longleftrightarrow\quad B\in\mathcal W.
\]
In particular, no finite set of identities defines membership in
$\mathcal W$ within $\mathcal F_{\mathrm{nil}}$.
\end{corollary}

\begin{proof}
Suppose that $\Sigma$ has the stated property, and choose $n\ge1$
bounding the number of variables in its identities.
Take the finite semiring $B_n=B(\HH)$ from the proof of
Theorem~\ref{thm:nilinterval}. It belongs to
$\mathcal F_{\mathrm{nil}}$ but not to $\mathcal W$.

Every subsemiring $C$ of $B_n$ generated by at most $n$ elements
belongs to $\mathcal N\subseteq\mathcal W$. It also belongs to
$\mathcal F_{\mathrm{nil}}$. Indeed, it is finite and commutative,
it contains the absorbing top because it contains the square of each
generator, its addition is the restriction of the flat addition of
$B_n$, and the same bound on lengths of nonzero products applies in
$C$. The assumed equivalence gives $C\models\Sigma$.
Each assignment of the variables in an identity of $\Sigma$ to $B_n$
takes its values in such a subsemiring $C$. Therefore
$B_n\models\Sigma$. The equivalence would now imply
$B_n\in\mathcal W$, a contradiction.
\end{proof}

\subsection{Simultaneous obstructions for semirings of bounded order}

For $r\ge1$, let $\mathcal K_r$ contain one representative of every
isomorphism type of ai-semiring of cardinality at most $r$, and put
$\mathcal U_r=\Var(\mathcal K_r)$. There are finitely many possible
operation tables, so $\mathcal K_r$ is finite and $\mathcal U_r$ is
generated by their finite direct product.

\begin{theorem}\label{thm:uniformsize}
For every pair of positive integers $n,r$ there exist positive
integers $d,L,k_0$ with the following property.
For every $k\ge k_0$ and every sufficiently large $N$ satisfying
$N\equiv k\pmod L$, there are a finite commutative flat semiring $B$
and an identity $\sigma$ such that:
\begin{enumerate}[$(1)$]
\item every subsemiring of $B$ generated by at most $n$ elements
belongs to $\mathcal N$;
\item every element of $B$ has square zero, and its multiplicative
reduct has nilpotency index exactly $k+1$;
\item every ai-semiring of cardinality at most $r$ satisfies
$\sigma$, whereas $B$ does not;
\item the cardinality of $B$ satisfies
\[
 N+2\ \le\ |B|\ \le\
 2+(2^k-2)\frac{\binom Nd}{\binom kd}.
\]
\end{enumerate}
Here the threshold for $N$ may depend on $k$.
In particular, $B\notin\mathcal U_r$, and for each fixed $k\ge k_0$
such semirings have arbitrarily large cardinality.
\end{theorem}

\begin{proof}
Let $A_r$ be the direct product of the members of $\mathcal K_r$,
and let $M_r$ be its multiplicative reduct.
Apply Lemma~\ref{lem:core} to $M_r$, obtaining
$D,L,\varepsilon$, and put $d=2D+4$.
Choose $k_0$ greater than
\[
 \max\{d,\,2(d-1),\,n(d-1),\,D\}.
\]
Fix $k\ge k_0$. The proof of Proposition~\ref{prop:separation}
applies to every sufficiently large $N$ with $N\equiv k\pmod L$.
Let $\HH=(V,\mathcal E)$ be its $k$-uniform hypergraph on $N$ vertices,
and set
\[
 B=B(\HH),\qquad
 \sigma:\quad t_{\HH}\approx t_{\HH}+q_{\HH}.
\]
Theorem~\ref{thm:local} proves (1).

All nonzero elements of $B$ have positive degree, and degrees add in
a nonzero product. No degree exceeds $k$, so every product of $k+1$
elements is zero. The product of the $k$ vertex elements of any block
is $p\ne0$. Hence the nilpotency index is exactly $k+1$.
Every square is zero by the construction of $B$.

The identity $\sigma$ holds in $\Pow(M_r)$ and hence in $A_r$ by
Proposition~\ref{prop:powersum}. Each member of $\mathcal K_r$ is an
image of $A_r$ under a coordinate projection, so it satisfies
$\sigma$. Its failure in $B$ proves (3).

For the upper cardinality estimate, every $d$-element subset of $V$
belongs to at most one block, because distinct blocks have intersection
of size less than $d$. Counting pairs consisting of a block and one
of its $d$-element subsets gives
\[
 |\mathcal E|\binom kd\le\binom Nd.
\]
Before the identifications defining $B$, each block contributes at
most $2^k-2$ nonempty proper subsets. Identifying full blocks gives
one element $p$, and the absorbing element contributes one more.
Consequently
\[
 |B|\le2+(2^k-2)|\mathcal E|
      \le2+(2^k-2)\frac{\binom Nd}{\binom kd}.
\]
For the lower estimate, all $N$ singleton subsets are distinct ordinary
elements: $k>d$ ensures that their degrees lie in the ordinary range.
They are distinct from both $p$ and the absorbing element, giving
$|B|\ge N+2$. Finally, there are arbitrarily large integers in the
required residue class, so the cardinalities are unbounded even
with $k$ fixed.
\end{proof}

The bounded-order comparison in this theorem differs from
\cite[Proposition~4.15]{GJRZ}. There the upper finite semirings are
required to belong to $\mathcal N$; here they range over all finite
ai-semirings of the given cardinality bound. The present construction
also imposes the local condition in part~\textup{(1)}.

The displayed estimate bounds the cardinalities of the constructed
algebras. Relating it to the equational complexity functions studied
in \cite{Jcomplexity} would additionally require quantitative control
of the separating identities and of the parameter choices in
Lemmas~\ref{lem:core} and~\ref{lem:random}.

\begin{corollary}\label{cor:boundedsizevarieties}
For every $r\ge3$, the variety $\mathcal U_r$ generated by all
ai-semirings of cardinality at most $r$ has infinite axiomatic rank.
\end{corollary}

\begin{proof}
The finite semiring $A_r$ generates $\mathcal U_r$, and
$S_7\in\mathcal U_r$ for $r\ge3$. Hence
$\mathcal N\subseteq\mathcal U_r=\Var(A_r)$, and
Theorem~\ref{thm:nilinterval} applies.
\end{proof}

\begin{corollary}\label{cor:exactorder}
Let $\mathcal T_r$ be the variety generated by all ai-semirings of
cardinality exactly $r$. Then $\mathcal T_r=\mathcal U_r$ for every
$r\ge1$. This variety is finitely based if and only if $r\le2$;
for every $r\ge3$ it has infinite axiomatic rank.
\end{corollary}

\begin{proof}
Given an ai-semiring $C$, adjoin a fresh element $\bot$ and prescribe
\[
 c+\bot=\bot+c=c,\qquad c\bot=\bot c=\bot
 \quad(c\in C\cup\{\bot\}).
\]
The enlarged addition is a semilattice operation with new identity,
and the enlarged multiplication is associative with new zero.
For distributivity, if the outside factor is $\bot$, both sides
are $\bot$; if a summand is $\bot$, its zero product disappears
from the sum on the other side. All other cases are the original
distributive laws in $C$. Thus the extension is an ai-semiring
of cardinality $|C|+1$ containing $C$ as a binary subsemiring.
Repeating the construction embeds every ai-semiring of cardinality
at most $r$ into one of cardinality exactly $r$.
Consequently $\mathcal U_r\subseteq\mathcal T_r$; the reverse
inclusion follows from the definitions.

For $r\ge3$, Corollary~\ref{cor:boundedsizevarieties} gives infinite
axiomatic rank. For $r=1$, the variety is trivial and is defined by
$x\approx y$. For $r=2$, finite basability is a result of Shao and
Ren~\cite{SR}, who proved that every subvariety of the variety
generated by all two-element ai-semirings has a finite basis.
\end{proof}

The infinite axiomatic rank assertion for $r=3$ recovers the
ai-semiring part of Wang, Ju and Wang~\cite{WJWthree}.
Corollary~\ref{cor:exactorder} extends this conclusion to every
order $r\ge3$ within the class of ai-semirings.

\subsection{Finite relatively free semirings}

For an ai-semiring variety $\mathcal L$, let $F_{\mathcal L}(r)$
denote its relatively free algebra on $r$ generators.

\begin{corollary}\label{cor:finitefree}
Let $\mathcal L$ be a locally finite ai-semiring variety containing
$S_7$. For every $r\ge2$, the finite semiring $F_{\mathcal L}(r)$ is
strongly nonfinitely based, and its generated variety has infinite
axiomatic rank.
\end{corollary}

\begin{proof}
The algebra $F_{\mathcal L}(r)$ is finite by local finiteness.
The elements $1,a$ generate $S_7$, because $a^2=0$.
By the universal property, mapping the first two free generators to
$1,a$ and all remaining generators to $1$ gives an onto homomorphism
$F_{\mathcal L}(r)\to S_7$. Thus
$S_7\in\Var(F_{\mathcal L}(r))$. Since
$\mathcal N\subseteq\Var(S_7)$, Theorem~\ref{thm:nilinterval}
gives infinite axiomatic rank for $\Var(F_{\mathcal L}(r))$.

If $A$ is any finite ai-semiring with
$F_{\mathcal L}(r)\in\Var(A)$, then
$S_7\in\Var(F_{\mathcal L}(r))\subseteq\Var(A)$.
Again Corollary~\ref{thm:main} shows that $A$ is nonfinitely based.
This is precisely the asserted strong nonfinite basability.
\end{proof}

An algebra is \emph{inherently nonfinitely based} if every locally
finite variety containing it is nonfinitely based. The following
example distinguishes this property from strong nonfinite basability
for $S_7$ and gives an explicit ambient variety for the preceding
corollary.

\begin{proposition}\label{prop:locallyfiniteambient}
Let $\mathcal L$ be defined by the ai-semiring axioms together with
\[
 xy\approx yx,\qquad x^2\approx x^3.
\]
Then $\mathcal L$ is finitely based and locally finite and contains
$S_7$. Every $F_{\mathcal L}(r)$ with $r\ge2$ is a finite strongly
nonfinitely based semiring. Moreover, $\mathcal L$ is not generated
by a finite semiring.
\end{proposition}

\begin{proof}
The displayed description is a finite identity basis.
To prove local finiteness, fix $r\ge1$ variables.
By distributivity, every term is a nonempty sum of nonempty words.
Commutativity allows each word to be written as a product of powers
of the variables, and $x^2\approx x^3$ reduces every exponent greater
than two to two. Thus there are at most $3^r-1$ possible nonempty
words, corresponding to exponent vectors in
$\{0,1,2\}^r\setminus\{(0,\ldots,0)\}$.
Commutativity and idempotence of addition reduce a polynomial to the
sum of a nonempty subset of these words. It follows that
\[
 |F_{\mathcal L}(r)|\le 2^{\,3^r-1}-1.
\]
Every $r$-generated member is an image of $F_{\mathcal L}(r)$, so
$\mathcal L$ is locally finite.

The tables of $S_7$ show that its multiplication is commutative and
that $x^2=x^3$ for each of its three elements. Hence $S_7\in\mathcal L$.
Corollary~\ref{cor:finitefree} applies to all $r\ge2$.
If $\mathcal L=\Var(A)$ for a finite semiring $A$, then
Corollary~\ref{thm:main} would make $\mathcal L$ nonfinitely based,
contrary to its defining finite basis.
\end{proof}

In particular, $S_7$ is strongly nonfinitely based but is not
inherently nonfinitely based.

\section{Naming neutral elements}
\label{sec:constants}

Write $\Var_1$ for generated varieties of unital ai-semirings in
$\{+,\cdot,1\}$, $\Var_0$ for generated varieties in
$\{+,\cdot,0\}$ satisfying
\[
 x+0\approx x,\qquad x0\approx0x\approx0,
\]
and $\Var_{0,1}$ for generated varieties in $\{+,\cdot,0,1\}$
satisfying both sets of neutral-element laws.

For an ai-semiring $C$, recall that $C^\bot$ is obtained by adjoining a new
element $\bot$, with
\[
 c+\bot=\bot+c=c,\qquad c\bot=\bot c=\bot
 \quad(c\in C^\bot).
\]
The symbol $0$ is interpreted by $\bot$.
Thus $S_7^\bot$ is the usual four-element extension $S_7^0$.
Its named additive zero is different from the absorbing top of its
three-element subsemiring $S_7$.
The ordinary nonfinite basis problem for this extension, and more
general zero adjunctions of commutative flat semirings, was considered
in \cite{WRZ,ZW}. Here the conclusions concern arbitrary finite
upper algebras and the specified signatures with constants.

The adjunctions below are changes of algebra, not changes of notation.
For example, if $C$ has an absorbing top $z$ with $z^2=z$, then
$C^\bot$ has the two distinct multiplicative idempotents $z,\bot$.
Its multiplicative reduct is therefore not nilpotent, even when
that of $C$ is nilpotent.

Preservation of variety membership under this adjunction in the
binary signature was established in
\cite[Corollary~2.2]{GRZcontinuum}. We give the corresponding direct
argument with the named constants retained.

\begin{lemma}\label{lem:zero-adjunction}
If $C\in\Var(S_7)$, then $C^\bot\in\Var_0(S_7^\bot)$.
If instead $C\in\Var_1(S_7)$, then
$C^\bot\in\Var_{0,1}(S_7^\bot)$.
\end{lemma}

\begin{proof}
In the first case, Birkhoff's theorem provides a subalgebra
$D\le S_7^I$ and an onto homomorphism $h:D\to C$.
We may take $I$ nonempty, using the representation of a relatively
free algebra by term functions if necessary.
Inside $(S_7^\bot)^I$, the algebra $D\cup\{\bot^I\}$ is isomorphic to
$D^\bot$: its new tuple is an additive identity and multiplicative
zero. The map $h$ extends to an onto homomorphism to $C^\bot$ by
mapping the new tuple to the new element, preserving the named
zero. This proves the first assertion.
For the second assertion, apply the same argument to a representation
in the unital signature; the subalgebra and the extended homomorphism
then preserve the named multiplicative identities as well.
\end{proof}

If $C$ is flat with absorbing top $z$ and every element has square
$z$, its \emph{flat unit adjunction}, when it is an ai-semiring,
is denoted by $C^1$. It is obtained by adjoining a new multiplicative
identity $1$ and stipulating $1+c=z$ for $c\in C$.

\begin{lemma}\label{lem:unit-adjunction}
Suppose that $C$ is a flat member of $\Var(S_7)$ and every square
in $C$ equals its absorbing top $z$. Then $C^1$ is an ai-semiring
and $C^1\in\Var_1(S_7)$.
\end{lemma}

\begin{proof}
If $C=\{z\}$, then $C^1$ is isomorphic to the unital subalgebra
$\{1,0\}$ of $S_7$. Assume that $C$ is nontrivial.
Choose an onto homomorphism $h:D\to C$ with $D\le S_7^I$ and
$I\ne\varnothing$, and let $e$ be the constant tuple $1$.
If $e\in D$, then $h(e)$ is a multiplicative identity of $C$
and is idempotent. The square condition forces $h(e)=z$, and
then every $c\in C$ satisfies $c=h(e)c=zc=z$, a contradiction.
Thus $e\notin D$.

The identity $1+x=x^2$ holds in $S_7$ when its unit is named.
It follows that $e+d=d^2\in D$ for $d\in D$, while $ed=de=d$.
Hence $D\cup\{e\}$ is a unital subalgebra of $S_7^I$.
Extend $h$ by sending $e$ to the new element $1$ of $C^1$.
For $d\in D$ the image of $e+d=d^2$ is
$h(d)^2=z=1+h(d)$, so mixed sums are preserved.
Mixed products are preserved because both added elements are units,
and the operations on $D$ are preserved by $h$.
The extension is onto the prescribed algebra $C^1$, proving that
$C^1$ is an ai-semiring and belongs to $\Var_1(S_7)$.
\end{proof}

\begin{lemma}\label{lem:witness-unit}
The flat unit adjunction $B(\HH)^1$ is an ai-semiring.
If every at most $n$-generated subsemiring of $B(\HH)$ belongs to
$\Var(S_7)$, then every at most $n$-generated unital subalgebra of
$B(\HH)^1$ belongs to $\Var_1(S_7)$.
\end{lemma}

\begin{proof}
The multiplicative semigroup of $B(\HH)$ is $0$-cancellative.
After adjoining a unit, the only possible new cancellation failure
with a nonzero common value would give $xb=x$ with nonzero
$x,b\in B(\HH)$. However, a nonzero product satisfies
$\deg(xb)=\deg(x)+\deg(b)>\deg(x)$.
Thus no such equality is possible. A unit as the fixed left factor
plainly causes no failure; the right cancellation argument is the
same. The enlarged semigroup is therefore $0$-cancellative.
As in Lemma~\ref{lem:semiring}, flat addition then distributes over
multiplication.

Take at most $n$ generators of $B(\HH)^1$. If they all equal $1$,
the generated unital subalgebra is $\{1\}$.
Otherwise let $C$ be the binary subsemiring generated by those lying
in $B(\HH)$. It contains the absorbing top, since it contains a
square. Multiplication by $1$ gives no new element, and addition
to $1$ gives that top. Thus the unital subalgebra in question is
$C^1$. The hypothesis gives $C\in\Var(S_7)$, and
Lemma~\ref{lem:unit-adjunction} gives $C^1\in\Var_1(S_7)$.
\end{proof}

\begin{theorem}\label{thm:namedconstants}
The unital semiring $S_7$ is strongly nonfinitely based in
$\{+,\cdot,1\}$. The semiring $S_7^\bot$ is strongly nonfinitely
based in each of $\{+,\cdot,0\}$ and $\{+,\cdot,0,1\}$, with the
indicated neutral elements named.
In each case, if $K$ denotes this small algebra and $A$ is a finite
ai-semiring with the indicated neutral elements and with
$K\in\Var(A)$, every variety
between $\Var(K)$ and $\Var(A)$ has infinite axiomatic rank.
Here all generated varieties and inclusions are taken in the
corresponding signature with constants.
\end{theorem}

\begin{proof}
Fix $n\ge1$ and a finite ai-semiring $A$ with the indicated neutral
elements in one of the three signatures. Apply
Proposition~\ref{prop:separation} to its
multiplicative reduct $M$, and put
\[
 B=B(\HH),\qquad \sigma=(t_{\HH}\approx t_{\HH}+q_{\HH}).
\]

First consider $\{+,\cdot,1\}$.
Equip $\Pow(M)$ with the named identity $\{1_A\}$.
The map $X\mapsto\sum X$ is an onto unital homomorphism to $A$,
so $A\models\sigma$. Use $B^1$ as the finite algebra violating
$\sigma$. The same assignment of vertex elements as in $B$ gives
this violation, because $B$ is a binary subsemiring of $B^1$ and
$\sigma$ has no constants. By Lemma~\ref{lem:witness-unit}, every
at most $n$-generated unital subalgebra belongs to $\Var_1(S_7)$.

Next consider $\{+,\cdot,0\}$.
Use the full power semiring $\mathcal P(M)$ with named zero
$\varnothing$. The identity $\sigma$ holds there.
When all variable values are nonempty, this is
Proposition~\ref{prop:separation}; if a variable has value
$\varnothing$, then $q_{\HH}$ has value $\varnothing$, since it
contains every variable, so the identity is automatic.
The map $X\mapsto\sum X$, with the empty sum interpreted as $0_A$,
is an onto homomorphism preserving the named zero.
Consequently $A\models\sigma$.

The finite algebra $B^\bot$ still violates $\sigma$ under the
vertex assignment. A subalgebra generated by at most $n$ elements
is either $\{\bot\}$ or $C^\bot$, where $C$ is the binary subsemiring
of $B$ generated by its non-$\bot$ generators. This description
follows from the neutral and absorbing laws for $\bot$.
Since $C\in\Var(S_7)$, Lemma~\ref{lem:zero-adjunction} gives
$C^\bot\in\Var_0(S_7^\bot)$.

Finally, consider $\{+,\cdot,0,1\}$.
Name both $\varnothing$ and $\{1_A\}$ in $\mathcal P(M)$.
The same homomorphism preserves both constants, so
$A\models\sigma$. Use $(B^1)^\bot$ as the finite algebra violating
$\sigma$. An at most $n$-generated subalgebra is obtained by taking
the unital subalgebra of $B^1$ generated by the non-$\bot$ generators
and adjoining $\bot$. If there are no such generators, the unital
subalgebra is $\{1\}$. Lemmas~\ref{lem:witness-unit} and
\ref{lem:zero-adjunction} show that the resulting subalgebra belongs
to $\Var_{0,1}(S_7^\bot)$.

In each case, let $K$ be the small algebra in the statement and let
$\Var(K)\subseteq\mathcal W\subseteq\Var(A)$ in the relevant
signature. The constructed finite algebra satisfies all identities
of $\mathcal W$ in at most $n$ variables: the values in any assignment
generate a subalgebra in $\Var(K)\subseteq\mathcal W$.
It does not belong to $\mathcal W$, because it fails $\sigma$,
which holds in $A$. As $n$ is arbitrary, $\mathcal W$ has infinite
axiomatic rank. Taking $\mathcal W=\Var(A)$ proves strong nonfinite
basability in each of the three signatures.
\end{proof}

\appendix
\section{Finite bases and the number of variables}\label{sec:rank}

For finite algebras in the signature $\{+,\cdot\}$, finite axiomatic
rank is equivalent to finite basability. We include the proof of this
standard fact.

\begin{proposition}\label{prop:rank}
Let $A$ be a finite ai-semiring. If all identities of $A$ follow from
its identities in at most $n$ variables, for some positive integer $n$,
then $A$ has a finite identity basis. Conversely, a finite identity
basis gives such a bound on the number of variables.
\end{proposition}

\begin{proof}
Let $F_n$ be the set of $n$-ary term functions of $A$. This set is
finite, since its elements are maps $A^n\to A$. For each $f\in F_n$,
choose a representing term $t_f(x_1,\ldots,x_n)$. Let $\pi_i$ denote
the $i$th projection. Consider the finite set $\Sigma_n$ consisting of
\[
 x_i\approx t_{\pi_i}\qquad(1\le i\le n)
\]
and
\[
 t_f\mathbin{\circ}t_g\approx t_{f\mathbin{\circ}g}
 \qquad(f,g\in F_n,\quad \circ\in\{+,\cdot\}),
\]
where $f\mathbin{\circ}g$ is the pointwise operation on term functions.
All these identities hold in $A$.

Induction on the formation of a term $t$ proves
$\Sigma_n\vdash t\approx t_{\widehat t}$, where $\widehat t$ is the
term function induced by $t$. Thus, if $u\approx v$ is an identity
of $A$ in the variables $x_1,\ldots,x_n$, then
$\widehat u=\widehat v$ and $\Sigma_n\vdash u\approx v$.
Renaming variables shows that $\Sigma_n$ implies every identity of
$A$ in at most $n$ variables. Under the hypothesis, it is therefore
a finite basis for $\Id(A)$.

Conversely, the maximum number of variables occurring in a finite
basis gives the required bound.
\end{proof}

Since the ai-semiring axioms form a finite set, finite basability
relative to ai-semirings and finite basability in the signature
$\{+,\cdot\}$ are equivalent.

\end{document}